\documentclass{amsart}
\usepackage{amsmath,amssymb,graphicx,hyperref,latexsym,url}

\graphicspath{{figures/}}

\newtheorem{theorem}{Theorem}
\newtheorem{corollary}[theorem]{Corollary}
\newtheorem{lemma}[theorem]{Lemma}
\newtheorem{proposition}[theorem]{Proposition}

\theoremstyle{definition}
\newtheorem{definition}[theorem]{Definition}
\newtheorem*{example*}{Example}
\newtheorem*{remark*}{Remark}

\newcommand{\N}{\mathbb{N}}
\newcommand{\Q}{\mathbb{Q}}
\newcommand{\R}{\mathbb{R}}
\newcommand{\Z}{\mathbb{Z}}
\newcommand{\nequiv}{\mathrel{\not\equiv}}
\newcommand{\colonequal}{\mathrel{\mathop:}=}

\begin{document}

\title{$p$-adic asymptotic properties of constant-recursive sequences}

\author{Eric Rowland}
\thanks{The first author was supported by a Marie Curie Actions COFUND fellowship.}
\address{
	Department of Mathematics \\
	University of Liege \\
	4000 Li\`ege \\
	Belgium
}
\curraddr{
	Department of Mathematics \\
	Hofstra University \\
	Hempstead, NY \\
	USA
}

\author{Reem Yassawi}

\address{
	Department of Mathematics \\
	Trent University \\
	Peterborough, Ontario \\
	Canada
}

\date{November 6, 2016}

\begin{abstract}
In this article we study $p$-adic properties of sequences of integers (or $p$-adic integers) that satisfy a linear recurrence with constant coefficients.
For such a sequence, we give an explicit approximate twisted interpolation to $\Z_p$.
We then use this interpolation for two applications.
The first is that certain subsequences of constant-recursive sequences converge  $p$-adically.
The second is that the density of the residues  modulo $p^\alpha$ attained by a constant-recursive sequence converges, as $\alpha\rightarrow \infty$, to the Haar measure of a certain subset of $\Z_p$.
To illustrate these results, we determine some particular limits for the Fibonacci sequence.
\end{abstract}

\maketitle

\section{Introduction}\label{introduction}

Many integer sequences $s(n)_{n \geq 0}$ that arise in combinatorial and number theoretic settings have the property that $(s(n) \bmod p^\alpha)_{n \geq 0}$ is a $p$-automatic sequence for each $\alpha \geq 0$ \cite{Rowland--Yassawi, Rowland--Zeilberger}.
As $\alpha$ varies, automata that produce these sequences have natural relationships to each other; namely, an automaton for a sequence modulo $p^\alpha$ necessarily contains all information about the sequence modulo smaller powers of $p$.
However, there has not been a satisfactory way of letting $\alpha \to \infty$ and capturing information about all powers $p^\alpha$ simultaneously.
The inverse limit of these automata 
is a {\em profinite automaton}~\cite{Rowland--Yassawi substitutions}.
In this paper we study properties of this profinite automaton for a sequence $s(n)_{n \geq 0}$ satisfying a linear recurrence with constant coefficients, by interpolating subsequences to 
the $p$-adic integers $\Z_p$.
Namely, we are interested in $p$-adic limits of certain subsequences of $s(n)_{n \geq 0}$, as well as the limiting density of attained residues of $s(n)_{n \geq 0}$ modulo powers of $p$.

\begin{figure}
	\includegraphics[scale=.65]{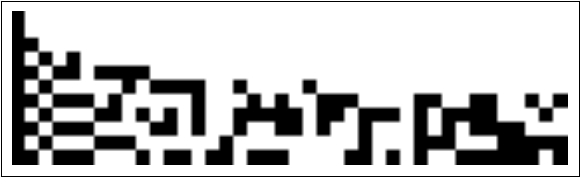} \quad
	\includegraphics[scale=.65]{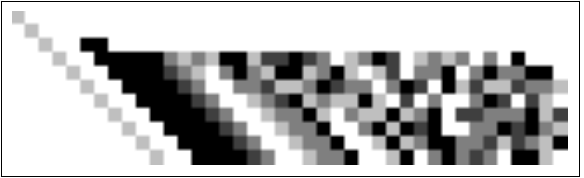} \quad
	\includegraphics[scale=.65]{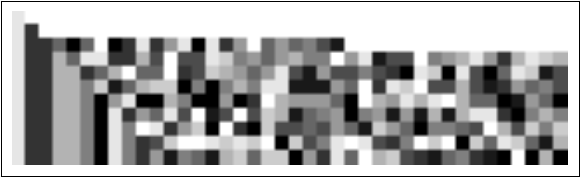}
	\caption{Base-$p$ digits of $F(p^n)$ for $p = 2$ (left), $p = 5$ (center), and $p = 11$ (right), for $n$ in the interval $0 \leq n \leq 10$.}
	\label{Fibonacci limits}
\end{figure}

For example, let $F(n)$ be the $n$th Fibonacci number.
Figure~\ref{Fibonacci limits} shows the first $40$ base-$p$ digits of $F(p^n)$ for $p \in \{2, 5, 11\}$ and $0 \leq n \leq 10$.
The digits $0$ through $p-1$ are rendered in grey levels ranging from white to black.
The $p=2$ array suggests that $\lim_{n \to \infty} F(2^n)$ does not exist in $\Z_2$ but that $\lim_{n \to \infty} F(2^{2 n})$ and $\lim_{n \to \infty} F(2^{2 n + 1})$ do.
The $p=5$ array suggests that $\lim_{n \to \infty} F(5^n) = 0$ in $\Z_5$, and the $p=11$ array suggests that $\lim_{n \to \infty} F(11^n)$ exists in $\Z_{11}$ and is non-zero.
Other limits of this nature appear elsewhere in the literature. For example, analogous limits for binomial coefficients were shown to exist by Davis~\cite{Davis}, and
limits of subsequences of the {\em gyration} sequence were used by Boyle, Lind, and Rudolph~\cite[Section~8]{Boyle--Lind--Rudolph} to obtain information about the automorphism group of a symbolic dynamical system. 

Regarding the Fibonacci sequence, Lenstra~\cite{Lenstra} showed that $F(n)_{n \geq 0}$ can be interpolated by an analytic function on the  {\em profinite integers}. 
For $p \neq 2$, Bihani, Sheppard, and Young~\cite{Bihani--Sheppard--Young} showed that $(a^n F(b n))_{n \geq 0}$ can be interpolated to $\Z_p$ by a hypergeometric function for some integers $a, b$.

A constant-recursive sequence
 cannot generally be interpolated to  $\Z_p$.
Namely, since $\mathbb N$ is dense in $\Z_p$ and $\Z_p$ is compact, a sequence $s(n)_{n\geq 0}$ can be interpolated to $\Z_p$ if and only if $(s(n) \bmod p^\alpha)_{n\geq 0}$ is purely periodic with period length equal to a power of $p$ for every $\alpha$.
However, we show in Theorem \ref{general power series} that
every constant-recursive sequence has an {\em approximate} twisted interpolation to $\Z_p$, as defined in Section~\ref{general sequence}. 
 In Theorem~\ref{converse}, we show that  in general this is the best we can hope for.
  We identify in Corollary~\ref{interpolations exist} a large family of  constant-recursive sequences that have twisted interpolations to $\Z_p$.
Amice and Fresnel give additional characterizations of sequences which have interpolations~\cite[Th\'eor\`eme~1]{Amice--Fresnel} and twisted interpolations~\cite[Section~3.2]{Amice--Fresnel} in terms of the analyticity of their generating functions.

Interpolation of this kind has been used previously to study arithmetic properties of constant-recursive sequences.
For example, the Skolem--Mahler--Lech theorem~\cite[Theorem~2.1]{Recurrence Sequences} for integer-valued constant-recursive sequences can be proved using interpolation.
More recently, Shu and Yao~\cite[Theorem~3]{Shu--Yao} implicitly used interpolation to characterize constant-recursive sequences of order $2$ whose sequence of $p$-adic valuations is $p$-regular. In this article we pay particular attention to the constants one must introduce, which allows us to make explicit the number of functions that comprise the twisted interpolation.

In Section~\ref{roots of unity} we give the necessary background in $p$-adic analysis.
In Section~\ref{general sequence} we discuss twisted interpolations to $\Z_p$ of a sequence satisfying a linear recurrence with constant coefficients.
In Section~\ref{applications} we apply interpolations to the computation of $p$-adic limits and limiting densities of attained residues.
In particular, we show in Theorem~\ref{limiting density} that  the limiting density of attained residues is the Haar measure of a certain set.
In Section~\ref{Fibonacci} we give a twisted interpolation for the Fibonacci sequence to $\Z_p$, we establish the limits suggested by Figure~\ref{Fibonacci limits}, and, in Theorem~\ref{Fibonacci p=11}, we determine the limiting density
of residues attained by the Fibonacci sequence modulo powers of $11$.

\section{Roots of unity in extensions of  $\Q_p$}\label{roots of unity}

We use several results about finite extensions of the field $\Q_p$ of $p$-adic numbers. A complete exposition of the following results can be found in \cite[Chapter 5]{Gouvea}.

If $a_i \in \{0, 1, \dots, p-1\}$ for all $i \geq k$ and $a_k \neq 0$, recall that the $p$-adic absolute value is defined on $\Q_p$ by  $|\sum_{i=k}^\infty a_i p^i|_p = p^{-k}$.

\begin{theorem}\label{background_absolute value}
Let $K/\Q_p$ be a finite extension of degree $d$. For $ x \in K$,
  let $M_ x$ be the matrix that corresponds to multiplication by  $ x$ in $K$, and define the multiplicative function $N_{K/ \Q_p}:K\rightarrow \Q_p$ as
 $N_{K/ \Q_p}( x) \colonequal \det M_ x$. 
\begin{enumerate}
 \item \cite[Corollary 5.3.2 and Theorem 5.3.5]{Gouvea} There is exactly one non-Archimedean absolute  value  $|\,\,|_p$ on $K$ extending the $p$-adic absolute value $|\,\,|_p$ on $\Q_p$, defined as 
 \[   | x|_p\colonequal \sqrt[d]{|    N_{K/\Q_p}( x)|_p}.   \]

 \item \cite[Proposition 5.4.2]{Gouvea} 
 Define the $p$-adic valuation $\nu_p: K\setminus \{ 0\} \rightarrow \Q$  as the unique number satisfying
 \[ | x|_p= p^{-\nu_p(x)}. \]
 Then the image of $\nu_p $  is $\frac{1}{e}\Z$, where $e$ is a divisor of $d$. \end{enumerate}
\end{theorem}

The value $e$ in Part (2) of Theorem~\ref{background_absolute value}
 is called the {\em ramification index} of the extension $K/\Q_p$.
 Akin to the special role of $p$ in $\Z_p$, we say $\pi\in K$ is a {\em uniformizer} if 
 $\nu_p(\pi) = 1/e$.

Given an extension $K/\Q_p$ with absolute value $|\,\,|_p$, let
\[ \mathcal O_{ K }\colonequal \{x \in K: |x|_p\leq 1 \} \] denote the unit ball in  $K$,
   and  let 
   \[ \mathcal U_{ K}\colonequal \{x \in K:  |x|_p <1\}\] denote its interior.
   Let $f\colonequal d/e$.

\begin{proposition}[{\cite[Propositions 5.4.5 and 5.4.6]{Gouvea}}]\label{properties}
Let $K/\Q_p$ be a finite extension of degree $d$, with ramification index $e$, and  $f=d/e$. Let $\pi \in K$ be a uniformizer. Then the following hold.
\begin{enumerate}
\item\label{principal}
$\mathcal U_K$ is a principal ideal of $\mathcal O_K$, and $\mathcal U_K = \pi \mathcal O_K$.
\item\label{field}
 The residue field $ \mathcal O_{ K   }   /\mathcal U_K$     is a finite field with $p^f$ elements.
\item\label{monic} If $\beta \in K$ is a root of a monic polynomial with coefficients in $\Z_p$, then $\beta\in \mathcal O_K$.
\item\label{expansion} Let $D=\{0,c_1, \ldots , c_{p^f -1} \} $ be a fixed set of representatives for the cosets of $\mathcal U_K$ in $\mathcal O_{ K   }$.
Then any $x\in  K$ has  a unique expansion $x= \sum_{j=-k}^{\infty} a_j \pi^{j}$ with each $a_i\in D$. 
\end{enumerate}
\end{proposition}

Part~\eqref{expansion} of Proposition~\ref{properties} indicates that elements of $K$ have a structure analogous to those of $\Q_p$, with $\pi$ playing the role of $p$.

Given an  extension $K/\Q_p$, the $p$-adic logarithm
\[
	\log_p x \colonequal \sum_{m \geq 1} (-1)^{m + 1} \frac{(x-1)^m}{m}
\]
converges for  $x \in 1 + \mathcal U_K$, i.e.\ for 
$x$ belonging to $\{ x \in \mathcal O_K: |x-1|_p<1    \}$.
The $p$-adic exponential function
\[
	\exp_p x \colonequal \sum_{m \geq 0} \frac{x^m}{m!}
\]
converges for $x$ belonging to $\{x \in \mathcal O_K : |x|_p < p^{-1/(p-1)}\}$.
If $|x-1|_p < p^{-1/(p-1)}$ then 
\[
	x = \exp_p \log_p x.
\]
For details, see \cite[Section~5.5]{Gouvea}.

The next proposition guarantees the existence of certain roots of unity in $\mathcal O_K$.

 \begin{proposition}[{\cite[Corollary 5.4.9]{Gouvea}}]\label{enough roots}
Let $K/\Q_p$ be a finite extension of degree $d$, with ramification index $e$, and  $f=d/e$.
Then 
$\mathcal O_{ K  }^{\times}$ contains the cyclic group of $(p^f-1)$-st roots of unity.
  \end{proposition}

  The proof of Proposition~\ref{enough roots} involves the appropriate version of Hensel's lemma, and in particular it implies that each $(p^f-1)$-st root of unity belongs to a distinct residue class modulo $\pi$.
Since there are precisely $p^f$ residue classes modulo $\pi$, it follows that, for each $x \in \mathcal O_K$ such that $x \nequiv 0 \mod \pi$, there is a unique $(p^f-1)$-st root of unity congruent to $x$ modulo $\pi$; we define $\omega(x)$ to be this  root of unity. Note that $\omega(x)$ is independent of the choice of uniformizer.
For $p \neq 2$ and a $p$-adic integer $ x \in \Z_p \setminus p \Z_p$, $\omega( x)$ coincides with the \emph{Teichm\"uller representative} of $ x$, the $(p-1)$-st root of unity congruent to $ x$ modulo $p$.

\section{Interpolation of a constant-recursive sequence}\label{general sequence}
Let $\mathbb N= \{0,1,\ldots\}$ be the set of natural numbers.
Let $s(n)_{n \geq 0}$ be a sequence of $p$-adic integers satisfying a linear recurrence
\begin{equation}\label{linear recurrence}
	s(n + \ell) + a_{\ell  -1}s(n+\ell -1) + \dots + a_1 s(n + 1) + a_0 s(n) = 0
\end{equation}
with constant coefficients $a_i \in \Z_p$.
As discussed in Section~\ref{introduction}, in general $s(n)$ cannot be interpolated to $\Z_p$.

\begin{definition}\label{definition}
Let $p$ be a prime, and let  $q \geq 1$ be a power of $p$.
Let $s(n)_{n\geq 0}$ be a sequence of $p$-adic integers.
Suppose $\mathbb N= \bigcup_{j\in J} A_j$ is a finite partition of $\mathbb N$,
  with each $A_j$ dense in $r+q \Z_p$ for some $0\leq r\leq q -1$.
 Let  $K$ be a finite extension of $\Q_p$, and for each $j \in J$ let $s_j:\Z_p\rightarrow K$ be a continuous function.
\begin{itemize}
\item
 If $s(n) = s_j(n)$ for all $n\in A_j$ and $j\in J$, then we say that the family $\{(s_j, A_j):  j\in J \}$ is a {\em twisted interpolation} of $s(n)_{n\geq 0}$ to $\Z_p$.
 \item
 If there are non-negative constants $C, D$, with $D<1$,  such that $|s(n)-s_j(n)|_p\leq C D^n$  for all $n\in A_j$ and $j\in J$, then we say that $\{(s_j , A_j): j \in J\}$ is an {\em approximate twisted interpolation} of $s(n)_{n\geq 0}$ to $\Z_p$.
\end{itemize}
\end{definition}

In the case of a twisted interpolation, since $A_j$ is dense in $r+q \Z_p$, the function $s_j(x)$ is  the unique continuous function which agrees  with $s(n)$ on $A_j$.
Note that some authors refer to each of the functions $s_j$ as a twisted interpolation.  If all the functions $s_j$ are the same then we have an {\em interpolation}.
In this section we identify conditions that guarantee the existence of a twisted interpolation of $s(n)_{n\geq 0}$ to $\Z_p$. 
If  $s(n)_{n\geq 0}$ does not satisfy these conditions, we show that it can only  be approximately interpolated. 
The sets $A_j$ we will obtain are all of the form
\begin{equation}\label{special_sets}
	A_{i,r}\colonequal \{m \geq 0 : \text{$m \equiv i \mod p^{f}-1$ and $m \equiv r \mod q $}\}
\end{equation}
for some fixed $f$.
The  proof of the following lemma  follows directly from the Chinese remainder theorem.

\begin{lemma}\label{dense} Let $p$ be a prime, let $q \geq 1$ be a power of $p$, and let $f\geq 1$.
For each  $0\leq i \leq p^{f}-2$ and $0\leq r\leq q-1$, the set $A_{i,r}$
is dense in  $r+q \Z_p$.

 \end{lemma}
 
 We recall the classical interpolation of the Fibonacci numbers to $\mathbb R$.
Let $\phi = \frac{1 + \sqrt{5}}{2}$ and $\bar{\phi} = \frac{1 - \sqrt{5}}{2}$.
Using the generating function of the Fibonacci sequence,
the $n$th Fibonacci number $F(n)$ can be written using Binet's formula
 \[
	F(n)
	= \frac{\phi^n - \bar{\phi}^n}{\sqrt{5}}.
\]
Thus to obtain an interpolation of $F(n)_{n\geq 0}$ to $ \mathbb R$, it suffices to interpret
\[
	\frac{\phi^x - \bar{\phi}^x}{\sqrt{5}}
\]
for $x \in \R$.
We can write $\phi^n = (\exp \log \phi)^n = \exp (n \log \phi)$ since $\phi$ is positive, and  
it follows that $\phi^n$ is interpolated by $\exp (x \log \phi)$.

Because $\bar{\phi}$ is negative, we write 
\[
	\bar{\phi}^n = (-1)^n (-\bar{\phi})^n = (-1)^n (\exp \log(-\bar{\phi}))^n = (-1)^n \exp (n \log(-\bar{\phi})),
\]
and it remains to interpolate $(-1)^n$ to $\R$.
A common, but not unique, choice is $\cos(\pi x)$.
 Therefore $F(n)_{n\geq 0}$ is interpolated to $\R$ by the analytic function
\[
	F(x)=\frac{\exp (x \log \phi) - \cos(\pi x) \exp (x \log (-\bar{\phi}))}{\sqrt{5}}.
\]

 The main idea of this section  is to carry out such an interpolation to $\Q_p$ instead of $\R$. 
The $n$th term of the constant-recursive sequence $s(n)_{n \geq 0}$ can be written as a linear combination of terms of the form $n^j \beta^n$ in a suitable field extension, where $\beta$ runs over the roots of the characteristic polynomial $g(x) = x^\ell + \dots + a_1 x + a_0$, and $j$ runs over the integers from $0$ to $m_\beta - 1$, where $m_\beta$ is the multiplicity of $\beta$. In other words, we can write
\[ s(n) =  \sum_{\beta} c_\beta(n) \beta^n  \]
for some polynomials $c_\beta(x) \in K[x]$, where we sum over all roots $\beta$ of $g(x)$. The key step  is to be able to legitimately write $x=\exp_p \log_p x$ for some modified version of the roots $\beta$. Lemma~\ref{general domain} tells us how to do this.

\begin{lemma}\label{general domain}
Let $p$ be a prime.
Let $g(x) \in \Z_p[x]$ be a monic polynomial, and let $\beta$ be a root of $g(x)$ in a splitting field $K$ of $g(x)$ over $\Q_p$, with $|\beta|_p = 1$.
Let $e$ be the ramification index of $K/\Q_p$.
Let
\begin{equation}\label{definition of q}
	q =
	\begin{cases}
		1					& \text{if $e < p - 1$} \\
		p^{\lceil \log_p(e+1) \rceil}	& \text{if $e \geq p - 1$,}
	\end{cases}
\end{equation}
where here $\log_p$ is the real logarithm to base $p$.
Then $|(\frac{\beta}{\omega(\beta)})^q - 1|_p < p^{-1/(p-1)}$.
\end{lemma}

\begin{proof}
Let $\pi$ be a uniformizer.
Since $|\beta|_p=1$, by Proposition~\ref{enough roots} there exists a root of unity $\omega(\beta) \in \mathcal O_K$ which is congruent to $\beta$ modulo $\pi$.
We have
\[
	\left| \frac{\beta}{\omega(\beta)} - 1 \right|_p
	= \frac{|\beta - \omega(\beta)|_p}{|\omega(\beta)|_p}
	= |\beta - \omega(\beta)|_p.
\]

First suppose $e < p - 1$, so that $q = 1$.
Since $\beta \equiv \omega(\beta) \mod \pi$, we have
\[
	|\beta - \omega(\beta)|_p \leq  |\pi|_p = p^{-1/e} < p^{-1/(p-1)}
\]
as desired.
Now suppose $e \geq p - 1$, so that $q = p^{\lceil \log_p(e+1) \rceil}$.
Since $\log_p(e+1) \leq \lceil \log_p(e+1) \rceil$, our choice of $q$ implies $\pi^q \equiv 0 \mod p \pi$. 
Since $\beta \equiv \omega(\beta) \mod \pi$, Kummer's theorem then implies $\beta^q \equiv \omega(\beta)^q \mod p \pi$, so
\[
	|\beta^q - \omega(\beta)^q|_p \leq  |p \pi|_p = p^{-1 - 1/e} < p^{-1/(p-1)}
\]
since $1 + \frac{1}{e} > 1 \geq \frac{1}{p - 1}$.
\end{proof}

We make two remarks regarding Lemma~\ref{general domain}.
The first is that the case $e \geq p-1$ in Equation~\eqref{definition of q} does occur. For example, 
consider the sequence defined by $s(n + 3) = 2 s(n)$ and $s(0) = s(1) = s(2) = 1$.
Let $p = 3$. Then $K= \Q_3(\sqrt[3] 2)$,
$e = 3 \geq 2 = p-1$, and  $q=9$.

The second remark is that the value of $q$ given by Lemma~\ref{general domain} is not necessarily optimal.
For example, the extension $K = \Q_2$ of degree $1$ contains the square root of unity $-1$.
This root of unity is not included in those guaranteed by Proposition~\ref{enough roots}, but allowing $\beta \in \Z_2 \setminus 2 \Z_2$ to be divided by $1$ or $-1$ allows us to reduce the value of $q$ from $2$ to $1$.

We now state the main result of this section.

\begin{theorem}\label{general power series}
Let $p$ be a prime, and 
let $s(n)_{n \geq 0}$ be a constant-recursive sequence of $p$-adic integers with monic characteristic polynomial $g(x) \in \Z_p[x]$.
Then there exists an analytic approximate twisted interpolation of $s(n)_{n\geq 0}$ to $\Z_p$.
\end{theorem}

\begin{proof}
As in Lemma~\ref{general domain}, let $K$ be a degree-$d$ splitting field of $g(x)$ over $\Q_p$ with ramification index $e$,  and $f = d/e$. Let $q$ be defined as in Equation~\eqref{definition of q}. We have
$s(n) = \sum_{\beta} c_\beta(n) \beta^n$ for some $c_\beta(x) \in K[x]$.

We mimic what is done to interpolate  the Fibonacci numbers to $\R$.  
By Proposition~\ref{properties}, all roots of $g(x)$ lie in $\mathcal O_K$.
Let $n \geq 0$ and $0 \leq r \leq q - 1$.
For each root $\beta$ such that $|\beta|_p = 1$, we have
\begin{align*}
	\beta^{q n + r}
	&= \omega(\beta)^{q n} \beta^r \left(\tfrac{\beta}{\omega(\beta)}\right)^{q n}\\
	&= \omega(\beta)^{q n} \beta^r \left(\exp_p \log_p \left(\tfrac{\beta}{\omega(\beta)}\right)^{q n}\right) \\
	&= \omega(\beta)^{q n} \beta^r \exp_p\left(n \log_p \left(\tfrac{\beta}{\omega(\beta)}\right)^q\right)
\end{align*}
by Lemma~\ref{general domain}.
Therefore
\begin{align*}
	s(q n + r)
	&= \sum_\beta c_\beta(q n + r) \beta^{q n + r} \\
	&= \sum_{|\beta|_p <1} c_\beta(q n + r) \beta^{q n + r} + \sum_{|\beta|_p = 1} c_\beta(q n + r) \omega(\beta)^{q n} \beta^r \exp_p\left(n \log_p \left(\tfrac{\beta}{\omega(\beta)}\right)^q\right).
\end{align*}
We discard terms involving $\beta^{q n + r}$ where $|\beta|_p <1$ since these tend to $0$ quickly.
For the remaining terms, we must replace $\omega(\beta)^{q n}$ with a function defined on $\Z_p$.

When $n$ is restricted to a fixed residue class modulo $p^f - 1$, the expression $\omega(\beta)^{q n}$ is constant, and we can now define, for 
each $0 \leq i \leq p^f - 2$ and $0 \leq r \leq q-1$, 
\[
	s_{i,r}(q x + r)
	\colonequal \sum_{|\beta|_p = 1} c_\beta(q x + r) \omega(\beta)^{i - r} \beta^r \exp_p\left(x \log_p \left(\tfrac{\beta}{\omega(\beta)}\right)^q\right)
\]
for $x \in \Z_p$.

Recall $A_{i,r} = \{m \geq 0 : \text{$m \equiv i \mod p^f - 1$ and $m \equiv r \mod q$}\}$ for $0 \leq i \leq p^f - 2$ and $0 \leq r \leq q-1$.  
Then for $m\in A_{i,r}$, $s_{i,r}(m)$ agrees with the second sum in the expression for $s(m)$.

We claim that $\{(s_{i,r}, A_{i,r}) :    0 \leq i \leq p^f - 2$ and $0 \leq r \leq q-1\}$
is an analytic approximate twisted interpolation of $s(n)_{n\geq 0}$ to $\Z_p$. By Lemma~\ref{dense}, each set $A_{i,r}$ has the correct density property.
Since $|\log_p(\tfrac{\beta}{\omega(\beta)})^q|_p< p^{-1/(p-1)}$ for each $\beta$ satisfying $|\beta|_p = 1$, the expression
\[
	\exp_p\left(x \log_p \left(\tfrac{\beta}{\omega(\beta)}\right)^q\right)
\]
is well defined for $x \in \Z_p$.
Therefore the function $x \mapsto s_{i,r}(q x + r)$ is analytic on $\Z_p$.
Since each $c_\beta$ is continuous, and $\Z_p$ is compact, we can define $C = \max_{|\beta|<1 } \max_{x\in \Z_p}|c_\beta(x)|_p$.
Then 
 for $n \in A_{i,r}$ we have
\[
	| s(n) - s_{i,r}(n) |_p
	= \left| \sum_{|\beta|_p <1} c_\beta(n) \beta^{n} \right|_p
	\leq \max_{|\beta|_p < 1} |c_\beta(n) \beta^{n}|_p
	\leq C \left(\max_{|\beta|_p < 1} |\beta|_p\right)^{n},
\]
where we interpret a maximum over the empty set to be $0$, and $0^0$ to be $0$.
\end{proof}

\begin{remark*}
We do not use the fact that the functions $c_\beta$ are polynomials, but only that they are continuous. Hence the proof of Theorem~\ref{general power series} works more generally for any sequence $s(n)_{\geq 0}$ which can be written $s(n) = \sum_{\beta} c_\beta(n) \beta^n$ as a sum over a finite set $B \subset \mathcal O_K$, where the functions $c_\beta$ are arbitrary continuous functions.
\end{remark*}

\begin{example*}
 Let $s(0)=s(1)=s(2) =1$, and $s(n+3)= 3s(n+2)+2s(n+1)-6s(n)$.
Let $p = 2$.
Then the roots of the characteristic polynomial are $3$ and $\pm \sqrt 2$. Because of these last two roots, $s(n)\neq s_{i,r}(n)$, and Theorem~\ref{general power series} gives only an approximate twisted interpolation.
\end{example*}

The proof of Theorem~\ref{general power series} gives us sufficient conditions for a  constant-recursive sequence $s(n)_{n \geq 0}$  to have  an analytic twisted interpolation to $\Z_p$.
If   $g(x)=x^ \ell + \dots + a_1 x + a_0       \in \Z_p[x]$ is a monic characteristic polynomial for $s(n)_{n \geq 0}$, with $s(n) = \sum_{\beta} c_\beta(n) \beta^n$ and 
\[\{\beta: |\beta|_p<1 \mbox{ and } c_\beta \mbox{ is not the 0 polynomial} \}=\emptyset, \] then there exists an analytic twisted interpolation of $s(n)_{n \geq 0}$ to $\Z_p$.
In particular, since, up to a unit,  $a_0=\prod_{\beta}\beta$, we have the following corollary.

\begin{corollary}\label{interpolations exist}
Let $p$ be a prime, and
let $s(n)_{n \geq 0}$ be a constant-recursive sequence of $p$-adic integers with monic characteristic polynomial $x^ \ell + \dots + a_1 x + a_0  
\in \Z_p[x]$. If $|a_0|_p=1$, then there exists a twisted interpolation of $s(n)_{n \geq 0}$ to $\Z_p$. 
\end{corollary}

If all roots of $g(x)$ satisfy $|\beta|_p = 1$, then we can extend $s(n)_{n \geq 0}$ to a two-sided sequence $s(n)_{n \in \Z}$ of $p$-adic integers satisfying Recurrence~\eqref{linear recurrence}.
In this case, Theorem~\ref{general power series} implies that $s(n) = s_{i,r}(n)$ for all $n \in \Z$ such that $n \equiv i \mod p^f - 1$ and $n \equiv r \mod q$.
Additionally, we obtain the following corollary.
We continue to assume the hypotheses of Theorem~\ref{general power series}.

\begin{corollary}
If $e < p - 1$ and all roots of $g(x)$ satisfy $\beta \equiv 1 \mod \pi$, then $s(n)$ can be interpolated to $\Z_p$.
\end{corollary}

\begin{proof}
Since $\beta \equiv 1 \mod \pi$, we have $\omega(\beta) = 1$.
It follows from Theorem~\ref{general power series} that, for a fixed $r$, the functions $s_{i,r}(q x + r)$ coincide for all $i$.
Since $e < p - 1$, we have $q = 1$, and therefore the only value of $r$ is $r = 0$, so $s(n) = s_{0,0}(n)$ for all $n \geq 0$.
\end{proof}

\begin{example*}
Let $p \geq 5$.
Let $s(0) = s(1) = 1$, and let $s(n + 2) = 2 s(n + 1) + (p - 1) s(n)$.
Then $e \leq 2 < p - 1$, and the roots of $x^2 - 2 x - (p - 1)$ are congruent to $1$ modulo $\pi$.
Therefore $s(n)_{n \geq 0}$ can be interpolated to $\Z_p$.
\end{example*}

Our next result tells us that Theorem~\ref{general power series} is the best that we can hope for.
\begin{theorem}\label{converse}

Let $p$ be a prime and $K$ a finite extension of $\Q_p$. Suppose that $B$ is a nonempty finite set of elements of $K$ such that $|\beta|_p<1$ for each $\beta \in B$. For each $\beta\in B$, let $c_\beta:\Z_p\rightarrow K$ be a continuous function.  
For $n\in \mathbb N$, define $s(n)=\sum_{  \beta\in B} c_\beta(n) \beta^n$, and 
suppose that there is a twisted interpolation of
 $s(n)_{n \geq 0}$ to $\Z_p$.
Then $s(n) = 0$ for all $n\geq 0$.
\end{theorem}

\begin{proof}
Let $(s_j, A_j)$ be a twisted interpolation for
  $s (n)_{n\geq 0}$.
Let $x\in \Z_p \setminus \mathbb N$, and define $k_n \in \mathbb N$ by $k_n = (x \bmod p^n)$. 
Fix $j$.
The closure $\overline A_j$ of each partition element $A_j$ in $\Z_p$ satisfies $\overline A_j =r+q \Z_p$ for some $0\leq r \leq  q-1$, where $q$ is as in Definition~\ref{definition}; this implies that there exists $x\in \Z_p\setminus \mathbb N$ such that $k_n\in A_j$ for sufficiently large $n$.  Now fix any such $x \in \overline A_j\setminus A_j$. As $n\rightarrow \infty, $ the continuity of $s_j$ implies that $s(k_n)=s_j(k_n)\rightarrow s_j(x)$. On the other hand,   as $n\rightarrow \infty$,  $\beta^{k_n}\rightarrow 0$  for each $\beta$, and we have that  $s(k_n)\rightarrow 0$. Thus $s_j$ is identically zero on $\overline A_j\setminus A_j$.  
Since $\overline A_j =r+q \Z_p$, if $k\in A_j$ then there exists a sequence of elements $(x_n)_{n \geq 0}$ in $ \overline A_j \setminus A_j$, such that $x_n\rightarrow k$,
and now continuity of $s_j$ again tells us that $s(k)=0$.   
\end{proof}

\begin{example*}
Consider the sequence defined by $s(n + 2) = 2 s(n)$ and $s(0) = s(1) = 1$.
Let $p = 2$.
The roots of the characteristic polynomial are $\pm \sqrt{2}$.
Since $\sqrt{2}$ is a uniformizer of $\Q_2(\sqrt{2})/\Q_2$, Theorem~\ref{converse} tells us that there is no twisted   interpolation of $s(n)_{n\geq 0}$ to $\Z_p$.
\end{example*}

\section{Limits and distribution of residues}\label{applications}

In this section we describe two applications of Theorem~\ref{general power series}.
The first concerns $p$-adic limits of subsequences of constant-recursive sequences, such as the limits suggested by Figure~\ref{Fibonacci limits}.
The second concerns the density of residues modulo $p^\alpha$ attained by a constant-recursive sequence.

Using the power series of Theorem~\ref{general power series}, we can compute limits of $s(n)$ along sequences of points in $A_{i,r}$.

\begin{corollary}\label{general limit}
Let $a, b \in \Z$ with $a \geq 1$.
Under the hypotheses of Theorem~\ref{general power series}, the limit $\lim_{n \to \infty} s(a p^{f n} + b)$ exists in $\Z_p$ and is equal to
\[
	\lim_{n \to \infty} s(a p^{f n} + b)
	= \sum_{|\beta|_p = 1} c_\beta(b) \omega(\beta)^a \beta^b
\]
In particular, the value of this limit is algebraic over $\Q_p$.
\end{corollary}

We note that if the coefficients $a_{\ell-1}, \dots, a_0$ in Recurrence~\eqref{linear recurrence} are integers and $s(n)_{n\geq 0}$ is integer-valued, then the limit above is algebraic over $\Q$.

\begin{proof}
For sufficiently large $n$, we have $a p^{f n} + b \equiv a + b \mod p^f - 1$ and $a p^{f n} + b \equiv b \mod q$. 
Therefore
\[
	| s(a p^{f n} + b) - s_{a+b, (b \bmod q)}(a p^{f n} + b) |_p
	\leq C \left(\max_{|\beta|_p \neq 1} |\beta|_p\right)^{a p^{f n} + b}.
\]
As $n \to \infty$, the right side of the inequality tends to $0$, so we have
\begin{align*}
	\lim_{n \to \infty} s(a p^{f n} + b)
	&= \lim_{n \to \infty} s_{a+b, (b \bmod q)}(a p^{f n} + b) \\
	&= s_{a+b, (b \bmod q)}(b) \\
	&= \sum_{|\beta|_p = 1} c_\beta(b) \omega(\beta)^a \beta^b
\end{align*}
by continuity of $s_{i,r}(q x + r)$.
\end{proof}

Corollary~\ref{general limit} is a generalization of the fact that if $\beta \in \Z_p \setminus p \Z_p$ then $\lim_{n \to \infty} \beta^{p^n} = \omega(\beta)$.
This can be seen since $s(n) = \beta^n$ satisfies the recurrence $s(n + 1) = \beta s(n)$ of order $1$.
For example, see~\cite{Rowland}.

It also follows from Corollary~\ref{general limit} that the sequence
\[
	\left(\lim_{n \to \infty} s(a p^{f n} + b)\right)_{b \geq 0}
\]
of $p$-adic integers satisfies Recurrence~\eqref{linear recurrence}, the original recurrence satisfied by $s(n)_{n \geq 0}$.
Other limits, such as $\lim_{n \to \infty} s(a_2 p^{2 f n} + a_1 p^{f n} + b)$, can be computed similarly.

If all roots of $g(x)$ satisfy $|\beta|_p = 1$, then we can relax the hypotheses of Corollary~\ref{general limit} and allow $a$ to be an arbitrary integer, obtaining the same conclusion.
Additionally, we have the following.

\begin{corollary}
Under the hypotheses of Theorem~\ref{general power series}, if $a \in (p^f - 1) \Z$ and all roots of $g(x)$ satisfy $|\beta|_p = 1$, then we obtain the integer limit
\[
	\lim_{n \to \infty} s(a p^{f n} + b)
	= s(b).
\]
\end{corollary}

\begin{proof}
The computation in the proof of Corollary~\ref{general limit} shows that
\[
	\lim_{n \to \infty} s(a p^{f n} + b)
	= \sum_\beta c_\beta(b) \omega(\beta)^a \beta^b
	= \sum_\beta c_\beta(b) \beta^b
	= s(b). \qedhere
\]
\end{proof}

Our second application of the power series of Theorem~\ref{general power series} is
determining the value of the limiting density
\[
	\lim_{\alpha \to \infty} \frac{|\{s(n) \bmod p^\alpha : n \geq 0\}|}{p^\alpha}
\]
of attained residues.  This limit exists since the sequence of densities modulo powers of $p$ is a non-increasing sequence bounded below by $0$.
Let $\mu$ be the Haar measure on $\Z_p$ defined by $\mu(m + p^\alpha \Z_p) = p^{-\alpha}$.
For $f\geq 1$ and $q\geq1$ a power of $p$, recall the definition of the family of sets $A_{i,r}$ in Equation~\eqref{special_sets}.

\begin{theorem}\label{limiting density}
Let  $s(n)_{n\geq 0}$ be a sequence of $p$-adic integers with an approximate twisted interpolation $\{(s_{i,r},A_{i,r}): 0\leq i\leq p^f -2 \mbox{ and }0\leq r\leq q -1\}$.
Then 
\[
	\lim_{\alpha \to \infty} \frac{|\{s(n) \bmod p^\alpha : n \geq 0\}|}{p^\alpha}
	= \mu\!\left(\Z_p \cap \bigcup_{i,r} s_{i,r}(r + q \Z_p)\right).
\]
\end{theorem}

\begin{proof}
First, note that
\[
	|\{s(n) \bmod p^\alpha : n \geq 0\}|
	= \left|\overline{s(\N)} \bmod p^\alpha\right|.
\]
For $n\in A_{i,r}$, define $t(n)= s_{i,r}(n)$.
Let the extension $K$ and the constants $C, D$ be as in Definition~\ref{definition}.
 For all $n\geq 0$, we have $|s(n)-t(n)|_p\leq C D^n$,
so $\overline{s(\N)} = \Z_p \cap \overline{t(\N)}$.
Since $\frac{\left|\overline{s(\mathbb N)} \bmod p^\alpha\right|}{p^\alpha}$ and $\frac{\left|\Z_p \cap \overline{t(\mathbb N)} \bmod p^\alpha\right|}{p^\alpha}$ are non-increasing functions of $\alpha$, the limits exist and
\[
	\lim_{\alpha \to \infty} \frac{\left|\overline{s(\N)} \bmod p^\alpha\right|}{p^\alpha}
	= \lim_{\alpha \to \infty} \frac{\left|\Z_p \cap \overline{t(\N)} \bmod p^\alpha\right|}{p^\alpha}.
\]
Therefore it suffices to work with  $t(n)$. Note that in the case of a twisted interpolation, $s(n)=t(n)$ and we work in $\Z_p$, but in general $t(n)$ is an element of $\mathcal O_K$.

By Lemma~\ref{dense},  we have $r + q \Z_p=\overline{A_{i,r}} $  for each $i$ and $r$, so
\[
	\overline{t(\mathbb N)}
	= \bigcup_{i,r} \overline{s_{i,r}(A_{i,r})}
	= \bigcup_{i,r} s_{i,r}(r + q \Z_p).
\]

Note that  $\overline{t(\mathbb N)}= t(\Z_p)$ is  compact, and hence closed, in $K$.
It follows that $\Z_p \cap \overline{t(\mathbb N)}$ is closed in $\Z_p$.
The set $\mathcal \Z_p \setminus \overline{t(\mathbb N)}$ is open, so it is a countable union of cylinder sets, i.e.\ sets of the form $k + p^\beta \Z_p$, where $k \in \Z_p$.
Since $\Z_p$ is a countable union of cylinder sets, it follows that $\Z_p \cap \overline{t(\mathbb N)}$ is also a countable union of cylinder sets and is therefore measurable.
It follows that
\[
	\mu\!\left(\Z_p \cap \overline{t(\mathbb N)}\right)
	= \mu\!\left(\Z_p \cap \bigcup_{i,r} s_{i,r}(r + q \Z_p)\right).
\]

It remains to show that
\[
	\lim_{\alpha \to \infty} \frac{\left|\Z_p \cap \overline{t(\mathbb N)} \bmod p^\alpha\right|}{p^\alpha}
	= \mu\!\left(\Z_p \cap \overline{t(\mathbb N)}\right).
\]
Let
\[
	S_\alpha\colonequal\bigcup_{k\in \left(\Z_p \cap \overline{t(\mathbb N)} \bmod p^\alpha\right)} k+   p^\alpha \Z_p.
\]
Since $\Z_p \cap \overline{t(\mathbb N)} \subset S_\alpha$, it  follows that  $\frac{\left|\Z_p \cap \overline{t(\mathbb N)} \bmod p^\alpha\right|}{p^\alpha}
	\geq \mu\!\left(\Z_p \cap \overline{t(\mathbb N)}\right)$ for each $\alpha$,
 and so $\lim_{\alpha \to \infty} \frac{\left|\Z_p \cap \overline{t(\mathbb N)} \bmod p^\alpha\right|}{p^\alpha} \geq \mu\!\left(\Z_p \cap \overline{t(\mathbb N)}\right)$.

To establish the other inequality, we fix $\epsilon>0$ and we suppose that for some $\alpha$, 
	$ \frac{\left|\Z_p \cap \overline{t(\mathbb N)} \bmod p^\alpha\right|}{p^\alpha} > \mu\!\left(\Z_p \cap \overline{t(\mathbb N)}\right) + \epsilon$, i.e.\ $\mu\!\left(S_\alpha\setminus \overline{t(\mathbb N)} \right) > \epsilon$.
Since $S_\alpha\setminus \overline{t(\mathbb N)}$ is open, there exists a set  $T\subset  S_\alpha\setminus \overline{t(\mathbb N)}$, which is a finite union of cylinder sets, and whose $\mu$-mass is at least $\frac{\epsilon}{2}$.
There exists $\beta > \alpha$ such that $T$ is a union of cylinder sets all of which are of the form $k+   p^\beta \Z_p$. Then $\mu\!\left(S_\beta\right) \leq \mu\!\left(S_\alpha\right) - \frac{\epsilon}{2}$. If  $\mu\!\left(S_\beta\setminus \overline{s(\mathbb N)}\right) > \epsilon$,
we iterate this procedure until  we find a $\gamma$ with $\mu\!\left(S_\gamma \setminus \overline{s(\mathbb N)}\right) < \epsilon$, and hence 
$ \lim_{\alpha \to \infty} \frac{\left|\Z_p \cap \overline{t(\mathbb N)} \bmod p^\alpha\right|}{p^\alpha} \leq \mu\!\left(\Z_p \cap \overline{t(\mathbb N)}\right) +\epsilon$.
As this is true for any $\epsilon>0$, this completes the proof.
\end{proof}

We apply  Theorem \ref{limiting density} to  compute the
 limiting density of residues attained by the Fibonacci sequence modulo $11^\alpha$   in Theorem \ref{Fibonacci p=11}. We suspect that the method we use there generalizes to any $p$ and any constant-recursive sequence.

\section{The Fibonacci sequence}\label{Fibonacci}

In this section we apply the results from Sections~\ref{general sequence} and \ref{applications} to the Fibonacci sequence $F(n)_{n \geq 0}$, which satisfies
\[
	F(n + 2) - F(n + 1) - F(n) = 0
\]
with initial conditions $F(0) = 0$ and $F(1) = 1$.
Accordingly, we take $K = \Q_p(\phi)$, where $\phi$ is a root of $x^2 - x - 1$.
Let $\bar{\phi}$ be the other root.
Note that $\sqrt{5} \colonequal 2 \phi - 1 \in \Q_p(\phi)$, and in fact $\Q_p(\sqrt{5}) = \Q_p(\phi)$.
The ramification index of the extension $\Q_p(\phi)/\Q_p$ is as follows.

\begin{lemma}\label{root_5_extension}
Let $p$ be a prime, and let $d$ be the degree of the extension $\Q_p(\phi)/\Q_p$.
\begin{itemize}
\item If $p \equiv 1,4 \mod 5$, then $\phi\in \Q_p$, so $e = d = 1$.
\item If $p \equiv  2, 3 \mod 5$, then $\phi\not \in \Q_p$ and $e = 1$ and $d = 2$. 
\item If $p=5$,  then $\phi\not \in \Q_5$ and $e = d = 2$. 
\end{itemize}
\end{lemma}

For $p = 5$ we take the uniformizer to be $\pi = \sqrt{5}$.
For other primes we take $\pi = p$.
Throughout this section, $e$ denotes the ramification index of $\Q_p(\phi)/\Q_p$, as determined in Lemma~\ref{root_5_extension}, and $f = d/e$.

For primes $p \neq 2$ we obtain the following.

\begin{theorem}\label{pnot5or2}
Let $p\neq 2$ be a prime, and let $0 \leq i \leq p^f - 2$.
Define the function $F_i : \Z_p \to \Z_p$ by
\[
	F_i(x)
	=  \sum_{m \geq 0} \frac{\left(\omega(\phi)^i - (-1)^m \omega(\bar{\phi})^i\right) \left(\log_p \frac{\phi}{\omega(\phi)}\right)^m}{m! \sqrt{5}} x^m,
\]
and let $A_i = \{n \geq 0 : n \equiv i \mod p^f - 1\}$.
Then $\{(F_i, A_i) : 0 \leq i \leq p^f - 2\}$ is a twisted interpolation of $F(n)_{n \geq 0}$ to $\Z_p$.
\end{theorem}

\begin{proof}
Since $p \neq 2$, we have $q = 1$ by Equation~\eqref{definition of q}.
We have
\[
	F(n)
	= \frac{\phi^n - \bar{\phi}^n}{\sqrt{5}}
\]
for each integer $n \geq 0$.
The roots of $x^2 - x - 1$ satisfy $|\phi|_p = |\bar{\phi}|_p = 1$.
By Theorem~\ref{general power series},
\[
	\frac{\omega(\phi)^i \exp_p\left(x \log_p \tfrac{\phi}{\omega(\phi)}\right) - \omega(\bar{\phi})^i \exp_p\left(x \log_p \tfrac{\bar{\phi}}{\omega(\bar{\phi})}\right)}{\sqrt{5}}
\]
defines an analytic function on $\Z_p$ which agrees with $F(n)$ on $A_i$.
Expanding the power series for $\exp_p$ gives
\[
	\sum_{m \geq 0} \frac{\omega(\phi)^i \left(\log_p \frac{\phi}{\omega(\phi)}\right)^m - \omega(\bar{\phi})^i \left(\log_p \frac{\bar{\phi}}{\omega(\bar{\phi})}\right)^m}{m! \sqrt{5}} x^m.
\]
We claim that $\log_p \frac{\bar{\phi}}{\omega(\bar{\phi})} = -\log_p \frac{\phi}{\omega(\phi)}$.
Since $p \neq 2$, we have $-1 = \phi \cdot \bar{\phi} \equiv \omega(\phi) \omega(\bar{\phi}) \mod \pi$; since $-1$ and $\omega(\phi) \omega(\bar{\phi})$ are both $(p^f - 1)$-st roots of unity, this implies $\omega(\phi) \omega(\bar{\phi}) = -1$.
Therefore
\[
	\log_p \tfrac{\phi}{\omega(\phi)} + \log_p \tfrac{\bar{\phi}}{\omega(\bar{\phi})} = \log_p 1 = 0,
\]
and 
\[
	F_i(x)
	= \sum_{m \geq 0} \frac{\left(\omega(\phi)^i - (-1)^m \omega(\bar{\phi})^i\right) \left(\log_p \frac{\phi}{\omega(\phi)}\right)^m}{m! \sqrt{5}} x^m. \qedhere
\]
\end{proof}

For $p = 2$ one can also state a version of Theorem~\ref{pnot5or2}, where there are $6$ functions in the twisted interpolation since $q = p = 2$.

For $p = 5$ it turns out that $F_i(x)$ simplifies somewhat, allowing us to interpolate a twisted Fibonacci sequence to $\Z_5$.
Define the $p$-adic hyperbolic sine by
\[
	\sinh_p(x) \colonequal \frac{\exp_p(x) - \exp_p(-x)}{2} = \sum_{m \geq 0} \frac{1}{(2 m + 1)!} x^{2 m + 1}.
\]

\begin{corollary}\label{Fibonacci 5}
Let $p = 5$.
The function $F(n)/\omega(3)^n$ can be extended to an analytic function on $\Z_5$, namely
\[
	\frac{2}{\sqrt{5}} \sinh_5 \! \left(x \log_5 \tfrac{\phi}{\omega(3)}\right).
\]
\end{corollary}

\begin{proof}
One checks that $\phi \equiv \bar{\phi} \equiv 3 \mod \sqrt{5}$ in $\mathcal O_{\Q_5(\phi)}$, so that
$\omega(\phi) = \omega(\bar{\phi}) = \omega(3)$, and the coefficient of $x^m$ is $0$ for even $m$.
Therefore, for every integer $n$, we have
\[
	F(n)
	= \sum_{m \geq 0} \frac{2 \, \omega(3)^n \left(\log_5 \frac{\phi}{\omega(3)}\right)^{2 m + 1}}{(2 m + 1)! \sqrt{5}} n^{2 m + 1}
	= \frac{2 \, \omega(3)^n}{\sqrt{5}} \sinh_5 \! \left(n \log_5 \tfrac{\phi}{\omega(3)}\right). \qedhere
\]
\end{proof}

Bihani, Sheppard, and Young~\cite{Bihani--Sheppard--Young} similarly showed that $2^n F(n)$ can be extended to an analytic function on $\Z_5$, in this case by a hypergeometric series.

Since $\omega(3)^{5^n} = \omega(3)$ in $\Z_5$ for all $n \geq 0$, from Corollary~\ref{Fibonacci 5} we see that the coefficient of $x^0$ in the power series expansion of $\frac{2}{\sqrt{5}} \sinh_5 \! \left(x \log_5 \tfrac{\phi}{\omega(3)}\right)$ is $\lim_{n \to \infty} F(5^n) = 0$.
Moreover, the coefficient of $x^1$ is
\[
	\lim_{n \to \infty} \frac{F(5^n)}{5^n}
	= \frac{2 \, \omega(3)}{\sqrt{5}} \log_5 \tfrac{\phi}{\omega(3)},
\]
the $5$-adic digits of which comprise the diagonal stripes seen in Figure~\ref{Fibonacci limits}.
Other coefficients of this power series can be obtained as limits similarly.

Corollary~\ref{general limit} allows us to establish the other limits suggested by Figure~\ref{Fibonacci limits}.
For a prime $p$ and $a, b \in \Z$, we have
\[
	\lim_{n \to \infty} F(a p^{f n} + b)
	= \frac{\omega(\phi)^a \phi^b - \omega(\bar{\phi})^a \bar{\phi}^b}{\sqrt{5}}.
\]
For $p = 2$ one computes that $\lim_{n \to \infty} F(p^{2n})$ and $\lim_{n \to \infty} F(p^{2n+1})$ are equal to $\pm \sqrt{-\frac{3}{5}}$.
For $p = 11$ the limit $\lim_{n \to \infty} F(p^n)$ is a root of $5 x^2 + 5 x + 1$.

We now turn to an application of Theorem~\ref{limiting density}.
A number of authors have studied the distribution of residues of the Fibonacci sequence modulo $m$.
Burr~\cite{Burr} characterized the integers $m$ such that $(F(n) \bmod m)_{n \geq 0}$ contains all residue classes modulo $m$.
In particular, the Fibonacci numbers attain all residues modulo $3^\alpha$ and all residues modulo $5^\alpha$.

The limiting densities of attained residues modulo powers of other primes can be determined by Theorem~\ref{limiting density}.
We conclude the paper by determining the limiting density of residues for $p = 11$.
In this case, $f = d = 1$, so the twisted interpolation of the Fibonacci sequence to $\Z_{11}$ consists of $10$ functions  $F_0,\ldots , F_{9}$.
By Theorem~\ref{limiting density},
\[
	\lim_{\alpha \to \infty} \frac{|\{F(n) \bmod 11^\alpha : n \geq 0\}|}{11^\alpha}
	= \mu\!\left(\bigcup_{i = 0}^9 F_i(\Z_{11})\right).
\]
Therefore it suffices to determine $F_i(\Z_{11})$ for each $i$ in the interval $0 \leq i \leq 9$.

\begin{lemma}\label{Fibonacci p=11 i not 5}
Let $p = 11$, and let $0\leq i \leq 9$ such that $i\neq 5$.
Then $F_i(\Z_{11}) = (F(i) \bmod 11) + 11 \Z_{11}$.
\end{lemma}

\begin{proof}
We determine the set $F_i(\Z_{11})$ by decomposing $F_i(x)$ as the composition of two simpler functions.
Using $\omega(\phi) \omega(\bar{\phi}) = -1$ and $\log_{11} \frac{\bar{\phi}}{\omega(\bar{\phi})} = -\log_{11} \frac{\phi}{\omega(\phi)}$, we have from the proof of Theorem~\ref{pnot5or2} that
\begin{align*}
	F_i(x)
	&= \frac{\omega(\phi)^i \exp_{11}\left(x \log_{11} \tfrac{\phi}{\omega(\phi)}\right) - \omega(\bar{\phi})^i \exp_{11}\left(-x \log_{11} \tfrac{\phi}{\omega(\phi)}\right)}{\sqrt{5}} \\
	&= \frac{\omega(\phi)^i \exp_{11}\left(x \log_{11} \tfrac{\phi}{\omega(\phi)}\right) - (-1)^i \omega(\phi)^{-i} \exp_{11}\left(-x \log_{11} \tfrac{\phi}{\omega(\phi)}\right)}{\sqrt{5}} \\
	&= \frac{h_i(x) - (-1)^i h_i(x)^{-1}}{\sqrt{5}}
\end{align*}
where $h_i(x) = \omega(\phi)^i \exp_{11}\left(x \log_{11} \tfrac{\phi}{\omega(\phi)}\right)$.

One computes $\left|\log_{11} \tfrac{\phi}{\omega(\phi)}\right|_{11} = \frac{1}{11}$, so $\left(\log_{11} \tfrac{\phi}{\omega(\phi)}\right) \Z_{11} = 11 \Z_{11}$.
Since $\exp_{11}$ is an isomorphism from the additive group $11 \Z_{11}$ to the multiplicative group $1 + 11 \Z_{11}$, we have
\[
	h_i(\Z_{11}) =
	\omega(\phi)^i (1 + 11 \Z_{11}) = (\phi^i \bmod 11) + 11 \Z_{11}.
\]

It remains to show that the image of $(\phi^i \bmod 11) + 11 \Z_{11}$ under the function $y \mapsto \frac{1}{\sqrt{5}} (y - (-1)^i y^{-1})$  is $(F(i) \bmod 11) + 11 \Z_{11}$.
Let
\[
	z \in \left(\frac{\phi^i - (-1)^i \phi^{-i}}{\sqrt{5}} \bmod 11\right) + 11 \Z_{11} = (F(i) \bmod 11) + 11 \Z_{11}.
\]
We apply Hensel's lemma to show that there exists $y \in (\phi^i \bmod 11) + 11 \Z_{11}$ such that $\frac{1}{\sqrt{5}} (y - (-1)^i y^{-1}) = z$, or, equivalently, $y^2 - \sqrt{5} z y - (-1)^i = 0$.
From our choice of $z$, it is clear that $y_0 = \phi^i$ satisfies this polynomial equation modulo $11$.
Then we must check that $2 y_0 - \sqrt{5} z \nequiv 0 \mod 11$.
The ring $\Z_{11}$ contains two square roots of $5$; without loss of generality, choose $\sqrt{5} \equiv 7 \mod 11$.
Then $\phi \equiv 4 \mod 11$, so $2 y_0 - \sqrt{5} z \nequiv 0 \mod 11$ if and only if $2 \cdot 4^i - (4^i - (-1)^i 4^{-i}) \nequiv 0 \mod 11$, which is true since $i\neq 5$.
\end{proof}

\begin{figure}[h]
	\includegraphics[width=\textwidth]{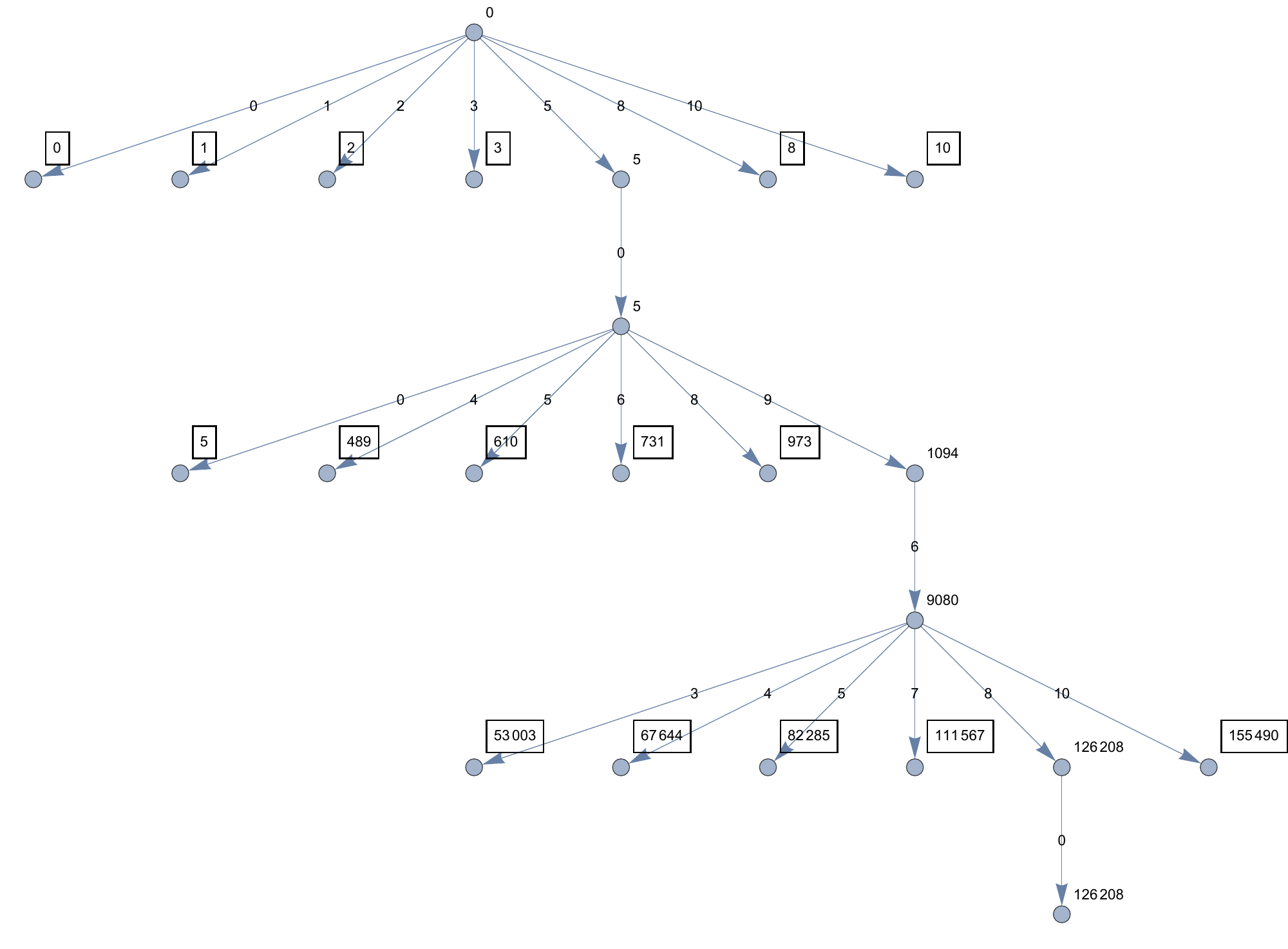}
	\caption{The tree of the residues attained by the Fibonacci sequence modulo small powers of $11$.}
	\label{Fibonacci tree}
\end{figure}

Figure~\ref{Fibonacci tree} shows the first several levels of the infinite rooted tree in which the vertices at level $\alpha$ consist of all residues $m$ modulo $11^\alpha$ such that $F(n) \equiv m \mod 11^\alpha$ for some $n \geq 0$.
Two vertices at consecutive levels $\alpha$ and $\alpha + 1$ are connected by an edge if the residue at level $\alpha + 1$ projects to the residue at level $\alpha$, and the edge is labeled with the extra base-$11$ digit in the residue at level $\alpha + 1$.
Framed residues represent full infinite $11$-ary subtrees: to simplify the diagram we suppress these full subtrees. 

It follows from Lemma~\ref{Fibonacci p=11 i not 5} that
\[
	\bigcup_{i\neq 5} F_i(\Z_{11}) = \bigcup_{ i\neq 5} (F(i) \bmod 11) + 11 \Z_{11} = \bigcup_{m \in \{0,1,2,3,8,10\}} m +11\Z_{11}.
\]
Accordingly, level $\alpha=1$ of Figure~\ref{Fibonacci tree} contains the residues $\{0,1,2,3,8,10\}$, and the outgoing edges from these vertices are suppressed since they consist of full $11$-ary subtrees.
Level $\alpha=1$ also contains the residue $5$; we will see that this residue has a unique residue modulo $11^2$ that projects onto it.

It remains to determine $F_5(\Z_{11})$.
We continue to choose $\sqrt{5} \equiv 7 \mod 11$.
We need to determine for which $z \in \Z_{11}$ the equation $y^2 - \sqrt{5} z y + 1 = 0$ has a solution in $\phi^5 + 11 \Z_{11} = 1 + 11 \Z_{11}$.
If $z = \frac{2}{\sqrt{5}}$, the equation becomes $(y - 1)^2 = 0$, which clearly has a solution in $1 + 11 \Z_{11}$.
Consequently, the tree in Figure~\ref{Fibonacci tree} contains an infinite path corresponding to the $11$-adic expansion of $\frac{2}{\sqrt{5}}$, and it is precisely along this path that more complicated branching occurs.

\begin{lemma}\label{Fibonacci_i=5}
Let $\alpha\geq 1$ and $j \in \{1, \dots, 10\}$.
Let $z \equiv \frac{2}{\sqrt{5}}   + j\, 11^{2\alpha} \mod 11^{2\alpha+1}$.
Let $f_z(y) \colonequal y^2 - \sqrt{5} z y + 1$.
\begin{itemize}
\item
If $j \in \{2, 6, 7, 8, 10\}$ then $f_z(y)$ has a root $y \in \Z_{11}$ satisfying $|y-1|_{11}<1/{11}^\alpha$.
\item
If $j \in \{1, 3, 4, 5, 9\}$ then $f_z(y)$ has no root in $\Z_{11}$.
\end{itemize}
\end{lemma}

\begin{proof}
We use the following version of Hensel's lemma: If there is an integer $a$ such that $|f_z(a)|_p < |f_z'(a)|_p^2$, then there is a unique $p$-adic integer $y$ such $f_z(y) = 0$ and $|y-a|_p < |f_z'(a)|_p$.

There are $6$ quadratic residues modulo $11$, namely $0$, $1$, $3$, $4$, $5$, and $9$.
If $j \in \{2, 6, 7, 8, 10\}$ then there exists $a' \in \{1, \dots, 10\}$ such that $a'^2 - \sqrt{5} j \equiv 0 \mod 11$.
We check that $a = 1 + 11^\alpha a'$ satisfies the conditions of Hensel's lemma.
We have
\begin{align*}
	f_z(a)
	&= a^2 - \sqrt{5} z a + 1 \\
	&\equiv a^2- \left(    2 + \sqrt{5} j \,11^{2\alpha} \right)      a+1 \mod 11^{2\alpha+1} \\
	&\equiv \left(a'^2 - \sqrt{5} j\right) 11^{2\alpha} \mod 11^{2\alpha+1} \\
	&\equiv 0 \mod 11^{2\alpha+1}.
\end{align*}
On the other hand, since $z\equiv   \frac{2}{\sqrt{5}} \mod 11^\alpha $ we have $f_z'(a) \equiv 2a - 2 \equiv 0 \mod 11^\alpha$, but $f_z'(a) \nequiv 0 \mod 11^{\alpha+1}$ since $a' \nequiv 0 \mod 11$.
It follows that
\[
	|f_z(a)|_{11} \leq \frac{1}{11^{2\alpha+1}} < \frac{1}{11^{2 \alpha}} = |f_z'(a)|_{11}^2.
\]

Since $z\equiv \frac{2}{\sqrt{5}} \mod 11^{2 \alpha}$, it follows that if $y^2-\sqrt{5}zy+1=0$ then $y \equiv 1 \mod 11^\alpha$.
Write $y = 1 + 11^\alpha a'$ for some $a' \in \Z_{11}$.
If $j \in \{1, 3, 4, 5, 9\}$ then $a'^2 - \sqrt{5} j \equiv 0 \mod 11$ has no solution in $a'$, so the computation above shows that $f_z(y) \nequiv 0 \mod 11^{2\alpha+1}$, which contradicts our assumption that $f_z(y) = 0$.
\end{proof}

\begin{lemma}\label{Fibonacci_rest}
Let $\alpha\geq 1$ and $j \in \{1, \dots, 10\}$.
Let $z \equiv \frac{2}{\sqrt{5}} + j \, 11^{2\alpha-1} \mod 11^{2\alpha}$. 
Then $f_z(y) \colonequal y^2 - \sqrt{5} z y + 1$ has no root in $\Z_{11}$.
\end{lemma}

\begin{proof}
The argument is similar to that of Lemma~\ref{Fibonacci_i=5}.
If $y^2-\sqrt{5}zy+1=0$ then $y \equiv 1 \mod 11^\alpha$.
Write $y = 1 + 11^\alpha a'$ for some $a' \in \Z_{11}$.
Then
\begin{align*}
	f_z(y)
	&= y^2 - \sqrt{5} z y + 1 \\
	&\equiv y^2 - \left(2 + \sqrt{5} j\,11^{2\alpha-1} \right) y + 1 \mod 11^{2\alpha} \\
	&\equiv -\sqrt{5} j \, 11^{2\alpha-1} \mod 11^{2\alpha} \\
	&\nequiv 0 \mod 11^{2\alpha},
\end{align*}
which contradicts our assumption that $f_z(y) = 0$.
\end{proof}

Lemmas~\ref{Fibonacci_i=5} and \ref{Fibonacci_rest} can be used to verify features of Figure~\ref{Fibonacci tree}. For example, letting $\alpha=1$ in Lemma~\ref{Fibonacci_rest} shows that the edge labeled $0$ is the only edge emanating from the residue $5$ modulo $11$ on level $\alpha = 1$.
The residue $5$ modulo $11^2$ on level $\alpha = 2$ has an emanating edge labeled $9$, since $\frac{2}{\sqrt{5}}\equiv 5+9 \cdot11^2 \mod 11^3$.
Letting $\alpha=1$ in Lemma~\ref{Fibonacci_i=5} shows that the other edges emanating from $5$ modulo $11^2$ are $9+\{2,6,7,8,10\} \bmod 11 = \{0,4,5,6,8  \}$.

\begin{theorem}\label{Fibonacci p=11}
The limiting density of residues attained by the Fibonacci sequence modulo $11^\alpha$ is
\[
	\lim_{\alpha \to \infty} \frac{|\{F(n) \bmod 11^\alpha : n \geq 0\}|}{11^\alpha}
	= \frac{145}{264}.
\]
\end{theorem}

\begin{proof}
It follows from Lemma~\ref{Fibonacci p=11 i not 5} that
\[
	\bigcup_{ i\neq 5   } F_i(\Z_{11}) = \bigcup_{m \in \{0, 1, 2, 3, 8, 10\}} m + 11 \Z_{11},
\]
which has measure $\frac{6}{11}$.
By Lemmas~\ref{Fibonacci_i=5} and \ref{Fibonacci_rest}, $F_5(\Z_{11})$ is a subset of $5 + 11^{2} \Z_{11}$ and so is disjoint from $F_i(\Z_{11})$ for each $i \neq 5$.
Moreover, it follows from these lemmas that $\mu(F_5(\Z_{11})) = \sum_{\alpha=1}^\infty \frac{5}{11^{2\alpha+1}} = \frac{1}{264}$, so that 
\[
	\mu\!\left(\bigcup_{i = 0}^9 F_i(\Z_{11})\right)
	= \frac{145}{264}. \qedhere
\]
\end{proof}

\section*{Acknowledgments}
The authors  thank Val\'erie Berth\'e for helpful discussions.
The second author thanks LIAFA, Universit\'e Paris-7 for 
its hospitality and support.

\end{document}